\documentclass[11pt]{amsart}

\usepackage[utf8]{inputenc}  

\usepackage{amsthm}
\usepackage{amsfonts}
\usepackage{amsmath}
\usepackage{amssymb}

\usepackage{mathtools} 
\usepackage{stmaryrd}
\usepackage{marvosym}

\usepackage{xcolor}
\definecolor{darkgreen}{rgb}{0,0.45,0}
\definecolor{darkred}{rgb}{0.75,0,0}
\definecolor{darkblue}{rgb}{0,0,0.6}
\usepackage[colorlinks,citecolor=darkgreen,linkcolor=darkred,urlcolor=darkblue]{hyperref}
\usepackage{breakurl}
\usepackage[mathscr]{eucal}
\usepackage{enumerate} 
\usepackage[capitalize]{cleveref}

\usepackage[margin=1.4in]{geometry}

\usepackage{tikz}
\usetikzlibrary{arrows}
\usepackage[all]{xy}
\xyoption{2cell}
\xyoption{curve}
\UseTwocells
\input{diagxy}  

\newbox\pbbox
\setbox\pbbox=\hbox{\xy \POS(65,0)\ar@{-} (0,0) \ar@{-} (65,65)\endxy}

\theoremstyle{plain}
\newtheorem{theorem}{Theorem}[section]

\newtheorem{proposition}[theorem]{Proposition}
\newtheorem{lemma}[theorem]{Lemma}

\theoremstyle{definition}

\newcommand{\Gpd}{\mathsf{Gpd}} 

\newcommand{\sSet}{\mathsf{sSet}}

\newcommand{\Aut}{\mathrm{Aut}} 
\newcommand{\Hom}{\operatorname{hom}} 
\newcommand{\Ob}{\operatorname{Ob}} 

\newcommand{\Gm}{\mathbb{G}_m}
\newcommand{\GLn}{\textbf{Gl}_n}

\newcommand{\PGLn}{\textbf{PGl}_n}
\newcommand{\PGLp}{\textbf{PGl}_p}

\newcommand{\BGLn}{B\textbf{Gl}_n}
\newcommand{\diag}{\operatorname{diag}}

\newcommand{\Bet}{B_{\text{et}}}
\newcommand{\BNis}{B_{\text{Nis}}}
\newcommand{\et}{\text{et}}
\newcommand{\Nis}{\text{Nis}}

\newcommand{\Iso}{\textbf{Iso}}
\newcommand{\Ext}{\textbf{Ext}}
\newcommand{\C}{\mathscr{C}}
\newcommand{\Mor}{\operatorname{Mor}}
\newcommand{\CenExt}{\mathrm{CenExt}}
\newcommand{\EZF}{\textbf{E}_{Z}F}
\newcommand{\sPreB}{\operatorname{s}\textbf{Pre}}

\newcommand{\Spec}{\operatorname{Spec}}

\newcommand{\sFvec}{\operatorname{s}(F-\textbf{vec})}
\newcommand{\sPreF}{\operatorname{s}\textbf{Pre}_{F}(Sm_k)}

\newcommand{\bbZ}{\mathbb{Z}} 

\renewcommand{\to}{\rightarrow} 
\newcommand{\iso}{\cong} 
\renewcommand{\equiv}{\simeq} 

\makeatletter
\newcommand{\colim@}[2]{%
  \vtop{\m@th\ialign{##\cr
    \hfil$#1\operator@font colim$\hfil\cr
    \noalign{\nointerlineskip\kern1.5\ex@}#2\cr
    \noalign{\nointerlineskip\kern-\ex@}\cr}}%
}
\renewcommand{\varprojlim}{%
  \mathop{\mathpalette\varlim@{\leftarrowfill@\scriptscriptstyle}}\nmlimits@
}
\renewcommand{\varinjlim}{%
  \mathop{\mathpalette\varlim@{\rightarrowfill@\scriptscriptstyle}}\nmlimits@
}
\makeatother

\begin{document}
\title{Central extensions and the classifying spaces of projective linear groups}

\author{Alexander Rolle}

\maketitle

\section*{Introduction}
The first result of this paper is a generalization of the fact that the set of isomorphism classes of central extensions of a group $G$ by an abelian group $A$ is in bijection with the elements of the cohomology group $H^2 (G,A)$. Letting $G$ be a presheaf of groupoids on a small site, and $A$ a sheaf of abelian groups, we use the work of Jardine on the theory of higher stacks to show that there is a bijection
\[
	H^2 (BG, A) \iso \CenExt(G,A)
\]
where $\CenExt(G,A)$ is an appropriate set of equivalence classes of ``central extensions'' of $G$ by $A$. Classes in $H^2 (BG, A)$ corresponding to central extensions of sheaves of groups can be interpreted as universal obstruction classes. The motivating example is the extension $\Gm \to \GLn \to \PGLn$ of sheaves of groups on the \'{e}tale site of a regular scheme, and in this case, our results recover the usual obstruction classes in the Brauer group of the base scheme.

In the second part of this paper, we consider instead the Nisnevich topology on the site $Sm_k$ of smooth schemes over a field $k$: for $G$ a presheaf of groups on this site, we show that Nisnevich $\PGLn$-torsors over the Nisnevich classifying space $\BNis G$ can be interpreted as homomorphisms $G \to \PGLn$. There is a universal obstruction class, which measures when a projective representation factors through a linear representation, in the motivic cohomology group $H^3 (\BNis \PGLn, \bbZ(1))$, by the results of the first section.

On the Nisnevich site $Sm_k$, there is another notion of classifying space that has been studied in motivic homotopy theory: this is the \'{e}tale classifying space $\Bet G$. The relationship between the motivic cohomology of $\BNis G$ and that of $\Bet G$ is quite interesting. According to a result of Morel-Voevodsky, the motivic cohomology of the \'{e}tale classifying space of a linear algebraic group $G$ recovers the Chow ring of the classifying space of $G$, in the sense of Totaro: $A^*_G \iso H^{2*} (\Bet G, \bbZ(*))$. In the final result of this paper, we use work of Vistoli on $A^*_{\PGLp}$ to show that, when $p$ is an odd prime, the canonical homomorphism in motivic cohomology
\[
	H^{2*} (\Bet \PGLp, \bbZ(*)) \to H^{2*} (\BNis \PGLp, \bbZ(*))
\]
is injective, over any field of characteristic zero containing a primitive $p^{th}$ root of unity. 

\section{Cohomology and central extensions}

The results of this section require significant parts of the theory of non-abelian cohomology developed by Jardine \cite[Chapter 9]{loc-hom-theory}, generalizing work of Giraud \cite{giraud} and Breen \cite{breen}. We begin by introducing the necessary parts of this theory.

Let $A$ be a sheaf of abelian groups on a site $\C$. There is an associated presheaf of 2-groups $\Iso(A)$ defined as follows: for $U$ an object of $\C$, the 1-morphisms of $\Iso(A)(U)$ are the isomorphisms
\[
	\alpha : A|_U \to A|_U
\]
of sheaves of groups on $\C / U$, and the 2-morphisms are given by
\begin{align*}
	\Iso(A)(U)(\alpha, \alpha) &= A(U) \; \; 
		\text{for all 1-morphisms $\alpha$} \; ; \\\
	\Iso(A)(U)(\alpha, \beta) &= \emptyset \; \;
		\text{for all $\alpha \neq \beta$} \; .
\end{align*}
Define a sheaf of 2-groups $A[2]$ on $\C$ such that $A[2](U)$ has only the identity 1-morphism, and has $A(U)$ for 2-morphisms.

Let $f: A[2] \to \Iso(A)$ be the map that includes $A[2]$ as a subobject, and define a retraction $g: \Iso(A) \to A[2]$ by
\[
	g(h: \alpha \to \alpha) = h : * \to * \; ,
\]
for any 2-morphism $h$ of $\Iso(A)(U)$.

The approach to non-abelian cohomology described in \cite[Chapter 9]{loc-hom-theory} makes heavy use of the Eilenberg-Mac Lane $\overline{W}$ functor. We will not need the details of the construction, which can be found in \cite[Section 9.3]{loc-hom-theory} or \cite{stevenson}, but we will use a few facts, which we now summarize.

Write $s_0\Gpd$ for the category of groupoids enriched in simplicial sets; then $\overline{W}$ is a functor $s_0\Gpd \to \sSet$. If $G$ is a groupoid enriched in simplicial sets, \cite[Corollary 9.39]{loc-hom-theory} says that there is a natural weak equivalence $d(BG) \to \overline{W}(G)$, where $d$ is the diagonal functor, and $BG$ is the bisimplicial set given by viewing $G$ as a simplicial groupoid, and applying the nerve functor $B$ pointwise.

The nerve also defines a functor $B : 2-\Gpd \to s_0\Gpd$. If $H$ is a 2-groupoid, then $BH$ is the groupoid enriched in simplicial sets with
\[
	\Ob(BH) = \Ob(H) \; \text{and} \; \Mor(BH) = B(\Mor(H)) \; .
\]
It is an abuse of notation, but we will write $\overline{W}(H)$ for the simplicial set given by first applying $B$ to the 2-groupoid $H$, and then applying $\overline{W}$.

These functors extend to the presheaf level by applying them sectionwise.

Because $A[2]$ is a sheaf of 2-groups, we can identify $B(A[2])$ with the simplicial sheaf of morphisms from the unique object to itself; this simplicial sheaf is just $BA$. By \cite[Corollary 9.39]{loc-hom-theory}, we have
\[
	\overline{W}A[2] \equiv d(BBA) \equiv K(A, 2) \; .
\]
There is a model structure on the category of presheaves of 2-groupoids on $\C$ such that a map $A \to B$ is a weak equivalence if and only if $\overline{W}A \to \overline{W}B$ is a local weak equivalence of simplicial presheaves; this is \cite[Theorem 9.57]{loc-hom-theory}.

The Eilenberg-Mac Lane functor $\overline{W}$ is part of a Quillen equivalence between simplicial presheaves and presheaves of groupoids enriched in simplicial sets \cite[Theorem 9.50]{loc-hom-theory}, and it follows from this that
\[
	[BG, \overline{W}(H)] \iso [G, H]_{2-\Gpd}
\]
for any presheaf of groupoids $G$ and for any presheaf of 2-groupoids $H$.

For presheaves of 2-groupoids $A,B$, write $h(A,B)$ for the category whose objects are diagrams
\[
	\xymatrix{
	&A &C \ar[l]_{\sim} \ar[r] &B
	}
\]
and whose morphisms are commutative diagrams
\[
	\xymatrix{
	& &C \ar[dl]_{\sim} \ar[dr] \ar[dd] &\\
	&A & &B\\
	& &D \ar[ul]^{\sim} \ar[ur] \\
	}
\]
This is the category of cocycles from $A$ to $B$. By \cite[Theorem 6.5]{loc-hom-theory}, we have
\[
	\pi_0 h(A,B) \iso [A,B]_{2-\Gpd} \; .
\]

So, there are isomorphisms
\begin{align*}
	H^2 (BG, A) &= [BG, K(A,2)] \\\
		&\iso [BG, \overline{W}A[2]] \\\
		&\iso [G, A[2]]_{2-\Gpd} \\\
		&\iso \pi_0 h(G, A[2]).
\end{align*}

As the map $f: A[2] \to \Iso(A)$ is a section, the induced map
\[
	f_* : \pi_0 h(G, A[2]) \to \pi_0 h(G, \Iso(A))
\]
is a monomorphism.

In \cite[Theorem 9.66]{loc-hom-theory}, Jardine constructs a bijection
\[
	 \pi_0 h(G, \Iso(A)) \iso \pi_0 \Ext(G, A)
\]
with the set of path components of a category $\Ext(G, A)$, which we'll now define.

If $p: G' \to G$ is a map of presheaves of groupoids, let im$(p)$ be the presheaf of groupoids that has the same objects as $G'$, and with im$(p)(x,y)$ given by the image of the function $G'(x,y) \to G(p(x), p(y))$. Say that a map $p: G' \to G$ is essentially surjective if the canonical map im$(p) \to G$ is a local weak equivalence of presheaves of groupoids, in the sense that the induced map of nerves is a local weak equivalence of simplicial presheaves.

Say that a kernel of a map $p: G' \to G$ is a diagram
\[
	\xymatrix{
	&K \ar[rr]^{j} \ar[dr] & &\Aut(G') \ar[dl]\\
	& &\Ob(G') & \\
	}
\]
where $K$ is a group object in the category of presheaves over $\Ob(G')$, and $j$ is a homomorphism of group objects over $\Ob(G')$, such that the following diagram is a pullback:
\[
	\xymatrix{
	&K \ar[r]^j \ar[d] &\Aut(G') \ar[d]^{p_*}\\
	&\Ob(G') \ar[r]_{p_* \cdot e} &\Aut(G)\\
	}
\]
Say that a kernel for $p$ in $A$ is a kernel $j: K \to \Aut(G')$ together with a map of presheaves $w: K \to A$ such that the induced map $K \to A \times \Ob(G')$ is a homomorphism of group objects over $\Ob(G')$ that induces an isomorphism of associated sheaves.

The objects of $\Ext(G, A)$ are triples $(p, j, w)$, where $p: G' \to G$ is essentially surjective, and $(j, w)$ is a choice of kernel in $A$ for $p$.

A morphism $\sigma : (p,j,w) \to (p',j',w')$ of this category is a local weak equivalence $\sigma: G' \to G''$ such that $p' \circ \sigma = p$, and $w' \circ \sigma_* = w$.

If $G$ is a presheaf of groups, and $p: G' \to G$ is an essentially surjective map of presheaves of groupoids, then $G'$ must be locally connected, i.e., a gerbe.

The map $K \to \Ob(G')$ is an example of what Jardine calls a family of presheaves of groups $\mathscr{F} \to S$ over the presheaf $S$, which is a group object in the category of presheaves over $S$. For an element $x \in S(U)$, the fibre $\mathscr{F}_x$ of the family $\mathscr{F}$ over $x$ is defined by the pullback diagram
\[
	\xymatrix{
	&\mathscr{F}_x \ar[r] \ar[d] &\mathscr{F}|_U \ar[d] \\
	&\ast \ar[r]_x &S|_U\\
	}
\]
If $(p: G' \to G, j,w)$ is an object of $\Ext(G, A)$, and $\alpha : x \to y$ is a morphism of $G'(U)$, then $\alpha$ defines an isomorphism $K_x \to K_y$ of presheaves of groups on $\C / U$ by conjugation; via $w$, we get an induced automorphism of $A|_U$.

Let $\CenExt(G, A) \subset \pi_0 \Ext(G, A)$ be the subset consisting of equivalence classes that have a representative $(p: G' \to G, j,w)$ with the following property: for all objects $U$ of $\C$, and for all morphisms $\alpha$ of $G'(U)$, the automorphism of $A|_U$ given by conjugation with $\alpha$ is the identity.

If $G$ happens to be a sheaf of groups, and $(p: G' \to G, j,w)$ is an object of $\Ext(G,A)$ such that $G'$ is a sheaf of groups, then the kernel of $p$ is necessarily isomorphic to $A$; the extension $(p,j,w)$ has the property of the last paragraph if and only if the kernel of $p_U$ is contained in the centre of $G'(U)$ for all objects $U$ of $\C$.

\begin{theorem} \label{class-of-cen-ext}
For any presheaf of groupoids $G$ on $\C$, and any sheaf of abelian groups $A$, there is a bijection
\[
	H^2 (BG, A) \iso \CenExt(G, A) \; .
\]
\end{theorem}

\begin{proof}
We will show that $\CenExt(G, A)$ is the image of the function
\[
	\xymatrix{
	&H^2 (BG, A) \ar[r]^{\iso} &\pi_0 h(G, A[2]) 
	\ar[r]^(.45){f_*} &\pi_0 h(G, \Iso(A)) 
	\ar[r]^{\iso} &\pi_0 \Ext(G, A).
	}
\]
First, we'll recall the definition of the bijection
\[
	\psi : \pi_0 h(G, \Iso(A)) \to \pi_0 \Ext(G, A)
\]
of \cite[Theorem 9.66]{loc-hom-theory}. Say
\[
	\xymatrix{
	&G &Z \ar[l]_(.4){g} \ar[r]^(.4){F} &\Iso(A)
	}
\]
is a cocycle from $G$ to $\Iso(A)$. Define a presheaf of 2-groupoids $\EZF$, whose objects are the objects of $Z$. The 1-morphisms $x \to y$ of $\EZF(U)$ are pairs $(\alpha, f)$ with $\alpha : x \to y$ a 1-morphism of $Z(U)$ and $f \in A(U)$. A 2-morphism $(\alpha, f) \to (\beta, g)$ of $\EZF(U)$ is a 2-morphism $h: \alpha \to \beta$ of $Z(U)$ such that $F(h) = g^{-1}f$.

If $\alpha : x \to y$ is a 1-morphism of $Z(U)$, then $F(\alpha)$ is a 1-morphism of $\Iso(A)(U)$, ie an automorphism of $A|_U$; write $\alpha_*$ for this automorphism. The composite of $(\alpha, f): x \to y$ and $(\beta, g): y \to z$ in $\EZF(U)$ is $(\beta \alpha, g\beta_*(f)): x \to z$.

Let $E_ZF = \pi_0(\EZF)$ be the presheaf of path component groupoids.

There is a map $\pi : \EZF \to Z$ which is the identity on objects, on 1-morphisms is $(\alpha, f) \mapsto \alpha$, and takes the 2-morphism $h : (\alpha, f) \to (\beta, g)$ to the underlying 2-morphism $h: \alpha \to \beta$. Write $g_*$ for the composite
\[
	\xymatrix{
	&E_ZF \ar[r]^{\pi_*} &\pi_0(Z) \ar[r]^(.55){\sim} &G.
	}
\]
Let $K(F) = A \times \Ob(E_ZF)$; define $j : K(F) \to \Aut(E_ZF)$ by letting $j_x : A|_U \to E_ZF_x$ be defined in sections by the rule $f \mapsto [(1_x, f)]$, for all $x \in \Ob(E_ZF)(U)$.

Let $w : K(F) \to A$ be projection; then $(j,w)$ is a kernel for $g_*$ in $A$.

The rule $(g,F) \mapsto (g_*, j, w)$ defines a functor $h(G, \Iso(A)) \to \Ext(G, A)$, and $\psi$ is the induced function on path components.

An element of $\pi_0 h(G, \Iso(A))$ is in the image of $f_*$ if and only if it has a representative $(g: Z \to G, F: Z \to \Iso(A))$ such that $F$ factors as $Z \to A[2] \subset \Iso(A)$. This is equivalent to the condition that, for all 1-morphisms $\alpha$ of $Z(U)$, $F(\alpha) = \alpha_*$ is the identity on $A|_U$. Say $(g,F)$ is such a cocycle. We will show that for any 1-morphism $\sigma$ of $E_ZF(U)$, the automorphism of $A|_U$ given by conjugation with $\sigma$ is the identity.

Let $\sigma = [(\alpha,f)]: x \to y$ be a 1-morphism of $E_ZF(U)$. Let $V \to U$ be an object of $\C / U$, and let $x \mapsto x'$ and $y \mapsto y'$ under $Z(U) \to Z(V)$. For any $g \in A(V)$,
\begin{align*}
	[(\alpha,f)][(1_{x'}, g)][(\alpha,f)]^{-1}
	&= [(\alpha,f)(1_{x'}, g)(\alpha^{-1},f^{-1})]\\\
	&= [(\alpha,f)(\alpha^{-1}, gf^{-1})]\\\
	&= [(1_{y'}, f\alpha_*(gf^{-1}))]\\\
	&= [(1_{y'}, g)]
\end{align*}
So $g \mapsto g$ and $\sigma$ induces the identity on $A|_U$.

We've proved that the image of our function $H^2 (BG, A) \to \pi_0 \Ext(G, A)$ is contained in $\CenExt(G, A)$.

To see the opposite inclusion, let's briefly recall the bijection
\[
	\phi: \pi_0 \Ext(G,A) \to \pi_0 h(G, \Iso(A)) \; ,
\]
inverse to $\psi$. If $(p: G' \to G, j,w)$ is an object of $\Ext(G,A)$, there is a cocycle
\[
	\xymatrix{
	&G &R(p) \ar[l]_q \ar[r]^(.4){F(p)} &\Iso(A).
	}
\]
The presheaf of 2-groupoids $R(p)$ has the same objects and 1-morphisms as $G'$, and there is a 2-morphism $\alpha \to \beta$ if and only if $p(\alpha) = p(\beta)$. For a morphism $\alpha$ of $G'(U)$, $F(p)(\alpha)$ is the automorphism of $A|_U$ defined by conjugation with $\alpha$.

We have $\phi[(p,j,w)] = [(q, F(p))]$. Clearly, then, if $[(p,j,w)]$ is an element of $\CenExt(G, A)$, then $\phi[(p,j,w)]$ is represented by a cocycle in the image of $f_*$, namely $(q, F(p))$. This completes the proof.
\end{proof}

The elements of $H^2 (BG, A)$ corresponding to central extensions of sheaves of groups can be interpreted as universal obstruction classes, in the following sense:

\begin{theorem}
Let $A \to G' \xrightarrow{p} G$ be a central extension, where $A, G'$ and $G$ are sheaves of groups. For any simplicial presheaf $X$, there is an exact sequence of pointed sets
\[
	\xymatrix{
	&H^1 (X, G') \ar[r]^{p_*} &H^1 (X, G) \ar[r]^{\pi} &H^2 (X, A),
	}
\]
and if $F \in H^1 (X, G)$, then $\pi(F) = F^*(c_p)$, where $c_p \in H^2 (BG, A)$ classifies the extension $A \to G' \to G$.
\end{theorem}

\begin{proof}
Because $A$ is central in $G'$, there is an induced action $BA \times BG' \to BG'$ of the simplicial sheaf of abelian groups $BA$ on the simplicial sheaf $BG'$. The Borel construction for this action is the bisimplicial sheaf $EBA \times_{BA} BG'$ with $(p,q)$-bisimplices given by the $q^{th}$ simplicial degree of the Borel construction for the action $A^{\times p} \times G'^{\times p} \to G'^{\times p}$.

The action of $A^{\times p}$ on $G'^{\times p}$ is free for all $p \geq 0$, so the map
\[
	EA^{\times p} \times_{A^{\times p}} G'^{\times p} \to G^{\times p}
\]
is a local weak equivalence. These maps induce a local weak equivalence from the diagonal
\[
	d(EBA \times_{BA} BG') \to BG \; .
\]
Moreover, there is a sequence of bisimplical sheaves
\[
	BG' \to EBA \times_{BA} BG' \to BBA,
\]
and, taking diagonals, the sequence
\[
	BG' \to d(EBA \times_{BA} BG') \to d(BBA)
\]
is a sectionwise fibre sequence, hence a local fibre sequence. In general, if $A \times X \to X$ is an action of a connected simplicial abelian group $A$ on a connected simplicial set $X$, then the sequences
\[
	X \to A^{\times p} \times X \to A^{\times p}
\]
are fibre sequences of connected simplicial sets, and so the sequence
\[
	X \to EA \times_A X \to BA
\]
of bisimplicial sets induces a fibre sequence of simplicial sets after taking diagonals, by a theorem of Bousfield and Friedlander, which appears as \cite[Theorem IV.4.9]{GJ}.

As $d(BBA) \equiv K(A,2)$, we have the exact sequence
\[
	\xymatrix{
	&H^1 (X, G') \ar[r]^{p_*} &H^1 (X, G) \ar[r]^{\pi} &H^2 (X, A).
	}
\]
This is exactly the argument of \cite[Example 9.11]{loc-hom-theory} in our case.

Say $F \in H^1 (X, G) = [X, BG]$. Then $\pi(F)$ is given by the diagram in the homotopy category
\[
	\xymatrix{
	&X \ar[r]^F &BG &d(EBA \times_{BA} BG') \ar[l]_(.7){\sim}
	\ar[r] &d(BBA).
	}
\]
To finish the proof, we need to show that the cocycle
\[
	BG \leftarrow d(EBA \times_{BA} BG') \to d(BBA) \tag{$\ast$}
\]
represents $c_p \in H^2 (BG, A)$.

First, note that
\[
	EBA \times_{BA} BG' \iso BBR(p),
\]
where $R(p)$ is the resolution 2-groupoid that appeared in the proof of \ref{class-of-cen-ext}. Using the natural weak equivalence $d(BH) \to \overline{W}(H)$, for $H$ a presheaf of groupoids enriched in simplicial sets, we have a pointwise equivalence from $(\ast)$ to the cocycle
\[
	\overline{W}G \leftarrow \overline{W}(R(p)) \to \overline{W}(A[2]).
	\tag{$\ast \ast$}
\]
The functor $\overline{W}$ has a left adjoint
\[
	\pi G : \sPreB \to \textbf{Pre}(2 - \mathsf{Gpd}),
\]
where $G$ is the loop groupoid functor to presheaves of groupoids enriched in simplicial sets, and $\pi$ is the fundamental groupoid functor to presheaves of 2-groupoids.

If $H$ is a groupoid enriched in simplicial sets, then $\pi(H)$ is the 2-groupoid with
\[
	\Ob(\pi(H)) = \Ob(H) \; \text{and} \; \Mor(\pi(H)) 
	= \pi(\Mor(H)) \; .
\]
This adjunction defines a Quillen equivalence between presheaves of 2-groupoids and the 2-equivalence model structure on simplicial presheaves of \cite[Theorem 5.49]{loc-hom-theory}. This follows from \cite[Proposition 9.59]{loc-hom-theory} and \cite[Proposition 9.61]{loc-hom-theory}. The 2-equivalence model structure has all monomorphisms for cofibrations, so that every object is cofibrant.

Now, for $M$ a presheaf of 2-groupoids and $\tilde{M}$ a fibrant model of $M$, we have natural weak equivalences
\[
	\xymatrix{
	&\pi G \overline{W}M \ar[r]^{\sim}
	&\pi G \overline{W} \tilde{M} \ar[r]^(.6){\sim} &\tilde{M},
	}
\]
as the functor $\overline{W}$ preserves weak equivalences. It follows from this that the cocycle obtained by applying $\pi G$ to $(\ast \ast)$ is pointwise equivalent to the cocycle
\[
	G \leftarrow R(p) \to A[2],
\]
which represents the class in $\pi_0 h(G, A[2])$ corresponding to the extension
\[
	A \to G' \to G.
\]
\end{proof}

The motivating example is the Brauer group:

On the \'{e}tale site $(Sch_S)_{\et}$ over a scheme $S$, the central extension of sheaves of groups
\[
	\Gm \to \GLn \xrightarrow{p} \PGLn
\]
corresponds to a class $c_p \in H^2_{\et} (B\PGLn, \Gm)$, and a $\PGLn$-torsor $F$ over $S$ lifts to a $\GLn$-torsor if and only if the class $F^*(c_p)$ vanishes in $H^2_{\et} (S, \Gm)$.

\section{The Nisnevich and \'{e}tale classifying spaces of projective linear groups}

\subsection{Preliminaries}

We'll often make use of the following contruction: if $G$ is a sheaf of groups on a site $\mathscr{C}$, let $G-\textbf{tors}$ be the presheaf of groupoids with $G-\textbf{tors}(U)$ the groupoid of $G$-torsors over $U$, for every object $U$ of $\mathscr{C}$. Then the obvious map
\[
	BG \to B(G-\textbf{tors})
\]
is a local weak equivalence, and $B(G-\textbf{tors})$ is injective fibrant \cite[Corollary 9.27]{loc-hom-theory}.

Let $k$ be a perfect field, and let $Sm_k$ be the category of smooth, separated $k$-schemes.

For $G$ a presheaf of groups on $Sm_k$, we'll use $BG$ to denote the Nisnevich classifying space of $G$. Following Morel and Voevodsky \cite{MV}, let $\Bet G$ be the Nisnevich homotopy type of an \'{e}tale fibrant model of $BG$. Explicitly, choose a map
\[
	j : BG \to F_{\et}(BG) \, ,
\]
where $j$ is an \'{e}tale local equivalence, and $F_{\et}(BG)$ is injective fibrant with respect to the \'{e}tale topology. Then $F_{\et}(BG)$ is a model of $\Bet G$.

As the name suggests, we have
\begin{align*}
	H^1_{\et} (X,G) &\iso \pi(X, B(G-\textbf{tors})_{\et})\\\
		&\iso [X, \Bet G]_{\Nis}
\end{align*}
where $\pi(-,-)$ denotes simplicial homotopy classes of maps.

In \cite[Lemma 4.1.18]{MV}, Morel and Voevodsky observe the following

\begin{proposition}
Let $G$ be a presheaf of groups. The map $BG \to \Bet G$ is a Nisnevich local equivalence if and only if $G$ is an \'{e}tale sheaf, and one of the following equivalent conditions holds:
\begin{enumerate}
	\item for any smooth scheme $S$ over $k$, one has 
	$H^1_{\Nis} (S,G) \iso H^1_{\et} (S,G)$.
	\item for any smooth scheme $S$ over $k$ and a point $x$ of $S$, 
	one has
	\[
		H^1_{\et} (\Spec(\mathcal{O}^h_{S,x}), G) = * \; .
	\]
\end{enumerate}
\end{proposition}

And, they point out \cite[Lemma 4.3.6]{MV} that general linear groups satisfy these conditions; so
\[
	B\GLn \to \Bet \GLn
\]
is a Nisnevich local equivalence for all $n > 0$.

\subsection{Nisnevich $\PGLn$-torsors}

If $S$ is a smooth scheme over a field, then the Nisnevich cohomology group $H^2_{\Nis} (S, \Gm)$ is zero, and so all Nisnevich $\PGLn$-torsors over $S$ lift to $\GLn$-torsors. The situation is more interesting over a simplicial scheme.

In the following, an inner automorphism of a presheaf of groups is an automorphism given by conjugation with a global section.

\begin{proposition} \label{tors-over-BG}
Let $G$ be a presheaf of groups on $Sm_k$, and let $H$ be a Nisnevich sheaf of groups. Then,
\[
	H^1_{\Nis} (BG, H) \iso \Hom (G, H) / \; \mathrm{inner} \; 
	\mathrm{automorphisms} \; \mathrm{of} \; H.
\]
\end{proposition}

\begin{proof}
We identify $H^1_{\Nis} (BG, H)$ with the set of maps in the homotopy category, with respect to the Nisnevich local equivalences, $[BG, BH]$. Because $B(H-\textbf{tors})$ is a fibrant model of $BH$, we have an isomorphism $[BG, BH] \iso \pi(BG, B(H-\textbf{tors}))$.

We'll begin by defining a function
\[
	\alpha : \pi(BG, B(H-\textbf{tors})) \to \Hom (G, H) / \; \mathrm{inner} 	\; \mathrm{automorphisms} \; \mathrm{of} \; H.
\]
Given a class in $\pi(BG, B(H-\textbf{tors}))$, choose a representative $\phi : BG \to B(H-\textbf{tors})$. Write $T$ for the $H$-torsor over $\Spec k$ corresponding to $\phi_0 : \Spec k \to \Ob(H-\textbf{tors})$. Because the groupoid of Nisnevich $H$-torsors over $\Spec k$ is contractible, we can choose an isomorphism $\tau_k : T \to H$, where $H$ is the trivial $H$-torsor over $\Spec k$; the choice of $\tau_k$ determines an isomorphism of $H$-torsors over $U$, $\tau_U : T_U \to H_U$, for every object $U$ of $Sm_k$. Furthermore, the choice of $\tau_k$ allows us to define a map $\psi : BG \to B(H-\textbf{tors})$, where $\psi_0 : \Spec k \to \Ob(H-\textbf{tors})$ corresponds to the trivial $H$-torsor $H$, and $\psi_1$ is defined by the rule $\psi_1(g) = \tau_U \phi_1(g) \tau_U^{-1}$ for every $g \in G(U)$. By construction, $\tau : \phi \Rightarrow \psi$ defines a simplicial homotopy. In simplicial degree 1, we have $\psi_1 : G \to \Aut(H) = H$; define
\[
	\alpha([\phi]) = [\psi_1].
\]
To see this is well-defined, let $\phi' : BG \to B(H-\textbf{tors})$ be a map with $\sigma : \phi \Rightarrow \phi'$ a simplicial homotopy. Write $T'$ for the $H$-torsor over $\Spec k$ corresponding to $\phi'_0$, and choose an isomorphism $\tau'_k : T' \to H$, which determines a simplicial homotopy $\tau' : \phi' \Rightarrow \psi'$. Let $\mu : \psi \Rightarrow \psi'$ be the composition $\tau' \circ \sigma \circ \tau^{-1}$. The homotopy $\mu$ gives an isomorphism $H \to H$ of $H$-torsors over $\Spec k$, i.e., an element $h \in H(k)$, and we have $\psi'_1 = h \psi_1 h^{-1}$, so that $\alpha$ is well-defined.

The function $\alpha$ is a bijection, as it has inverse
\[
	\beta : \Hom (G, H) / \; \mathrm{inner} 
	\; \mathrm{automorphisms} \; \mathrm{of} \; H 
	\to \pi(BG, B(H-\textbf{tors})),
\]
defined as follows: given a class in the domain, choose a representative $f : G \to H$, and let
\[
	\beta([f]) = [BG \xrightarrow{B(f)} BH \to B(H-\textbf{tors})].
\]
If $h \in H(k)$ and $f' = hfh^{-1}$, then $h$ defines a simplicial homotopy $B(f) \Rightarrow B(f')$, so $\beta$ is well-defined.
\end{proof}

So, for any presheaf of groups $G$, a Nisnevich $\PGLn$-torsor over $BG$ is given by a homomorphism $f : G \to \PGLn$, and $f$ lifts to a $\GLn$-torsor over $BG$ if and only if $f$ factors through the canonical map $p: \GLn \to \PGLn$.

By the results of Section 1, the central extension $\Gm \to \GLn \to \PGLn$ corresponds to an element $c_p$ of the group $H^2_{\Nis} (B\PGLn, \Gm)$, and a homomorphism $f : G \to \PGLn$ factors through $\GLn$ if and only if $f^*(c_p)$ vanishes in $H^2_{\Nis} (BG, \Gm)$. This is a statement about motivic cohomology, because of the isomorphism
\[
	H^2_{\Nis}(X, \Gm) \iso H^3 (X, \bbZ(1)),
\]
for any simplicial presheaf $X$.

In contrast, the group $H^3 (\Bet G, \bbZ(1))$ is zero if $G$ is a linear algebraic group, because $\Bet G$ has a model that is a sequential colimit of smooth schemes, by \cite[Proposition 4.2.6]{MV}.

In particular, the canonical homomorphism in motivic cohomology
\[
	H^* (\Bet \PGLn, \bbZ(*)) \to H^* (B_{\Nis} \PGLn, \bbZ(*))
\]
is not surjective.

\subsection{Chern classes}

Recall that we're using $BG$ to denote the Nisnevich classifying space of a presheaf of groups $G$ on the smooth site $Sm_k$.

If $k$ is a perfect field, the motivic cohomology of $\BGLn$ is a polynomial algebra over the cohomology of the base field \cite{pushin}:
\[
	H^*(\BGLn, \bbZ (*)) \iso H^* (k, \bbZ(*))[c_1, \dots , c_n]
\]
with $c_i \in H^{2i} (\BGLn, \bbZ(i))$.

If $G$ is a presheaf of groups on the Nisnevich site $Sm_k$ as before, and $f : G \to \GLn$ is a representation, then we can define the chern classes of $f$ to be
\[
	c_i(f) = f^*(c_i) \in H^{2i} (BG, \bbZ(i)) \; .
\]
As $B\GLn \equiv \Bet \GLn$, we can define chern classes in the motivic cohomology of $\Bet G$ in the same way. Because the canonical map in the homotopy category $BG \to \Bet G$ is natural in $G$, the homomorphism in motivic cohomology
\[
	H^{2*} (\Bet G, \bbZ(*)) \to H^{2*} (BG, \bbZ(*))
\]
takes chern classes to chern classes.

The goal of this section is to show that, when $p$ is an odd prime, this homomorphism
\[
	H^{2*} (\Bet \PGLp, \bbZ(*)) \to H^{2*} (B\PGLp, \bbZ(*))
\]
is injective, over any field of characteristic zero containing a primitive $p^{th}$ root of unity.

First, we need a lemma, which says that the motivic cohomology of a constant simplicial presheaf, with coefficients in a field, is as simple as possible. 

Let $X$ be a simplicial set, and let $F$ be any field. There is an adjunction
\[
	\Gamma^* : \sFvec \rightleftarrows \sPreF : \Gamma_*
\]
where $\Gamma^*$ is the constant presheaf functor, and $\Gamma_*$ is global sections.

\begin{lemma} \label{coh-of-constant}
Let $F$ be a field, and let $X$ be a simplicial set such that all singular cohomology groups $H^r (X,F)$ are finite-dimensional. Then,
\[
	H^* (\Gamma^* X, F(*)) \iso H^* (k, F(*)) \otimes H^*(X, F) \; ,
\]
where elements of $H^r (X,F)$ are seen as elements of the motivic cohomology group $H^r (\Gamma^* X, F(0))$.
\end{lemma}

\begin{proof}
The adjunction $\Gamma^* \dashv \Gamma_*$ is a Quillen adjunction for the injective model structure on $\sPreF$ and the usual model structure on $\sFvec$. So, we have
\begin{align*}
	H^p (\Gamma^* X, F(q)) &= [F(\Gamma^* X), F(q)[-p]] \\
		&\iso [FX, F(q)[-p](k)].
\end{align*}
For any simplicial $F$-vector spaces $C$ and $D$,
\[
	[C,D] \iso \prod_{n \geq 0} \operatorname{hom}(H_n (C), H_n (D)) \; .
\]
For $n \geq 0$, we have
\[
	H_n (F(q)[-p](k)) = H^{p-n}(k, F(q)),
\]
and so we have
\[
	H^p (\Gamma^* X, F(q)) \iso \prod_{n \geq 0} 
	\text{hom}(H_n (X, F), H^{p-n} (k, F(q))) \; .
\]
As $H^r (k, F(0)) \iso F$ if $r = 0$, and is zero otherwise, we have 
\[
	H^* (\Gamma^*X, F(0)) \iso H^* (X,F) \; .
\]

If $V,W$ are $F$-vector spaces with $V$ finite-dimensional, then the canonical map
\[
	V^{\vee} \otimes W \to \Hom(V,W)
\]
is an isomorphism. So, by our assumptions on $X$, we have
\[
	H^p (\Gamma^* X, F(q)) \iso \prod_{n \geq 0} 
	H^n (X, F) \otimes H^{p-n} (k, F(q)) \; .
\]
\end{proof}

In order to prove our result for $\PGLp$, we'll need its analogue for the group scheme of roots of unity $\mu_p$.

\begin{lemma} \label{chow-of-mu}
Let $p$ be prime. Over any perfect field $k$ containing a primitive $p^{th}$ root of unity, the homomorphism
\[
	H^{2*} (\Bet \, \mu_p, \bbZ(*)) \to H^{2*} (B \mu_p, \bbZ(*))
\]
is injective.
\end{lemma}

\begin{proof}
The Chow ring of $\Bet \, \mu_p$ is generated by the first chern class $t$ of the embedding $\mu_p \subset \Gm$ (\cite[p190]{Vistoli}):
\[
	H^{2*} (\Bet \, \mu_p, \bbZ(*)) \iso \bbZ[t]/(p \cdot t) \; .
\]
Furthermore,
\[
	H^2 (B\mu_p, \bbZ(1)) \iso \Hom (\mu_p, \Gm) \iso \bbZ / p \; .
\]
Let $c$ denote the first chern class of $\mu_p \subset \Gm$ in $H^2 (B\mu_p, \bbZ(1))$: we need to show that $c$ has infinite multiplicative order in $H^{2*} (B\mu_p, \bbZ(*))$. For this we use $\bbZ/p$ coefficients.

First, note that
\[
	H^0 (k, \bbZ/p(i)) \iso H^0_{\et} (k, \mu_p^{\otimes i})
	\iso \bbZ/p \; ,
\]
and it's easy to see that a generator $x \in H^0 (k, \bbZ/p(1))$ has infinite multiplicative order.

As $k$ contains a primitive $p^{th}$ root of unity, the obvious map
\[
	\Gamma^* BC_p \to B\mu_p
\]
is a Nisnevich local equivalence, where $C_p$ denotes the cyclic group with $p$ elements. By \ref{coh-of-constant},
\[
	H^2 (\Gamma^* BC_p, \bbZ/p(0)) \iso H^2 (BC_p, \bbZ/p) 
	\iso \bbZ/p \; ,
\]
and a generator $y$ of this group has infinite multiplicative order in $H^* (\Gamma^* BC_p, \bbZ/p(0))$. Again by \ref{coh-of-constant}, $x \cdot y \in H^2 (\Gamma^* BC_p, \bbZ/p(1))$ has infinite multiplicative order.

One checks that, under the map
\[
	H^2 (\Gamma^* BC_p, \bbZ(1)) \to H^2 (\Gamma^* BC_p, \bbZ/p(1))
\]
we have $c \mapsto x \cdot y$, so that $c$ has infinite multiplicative order as well.
\end{proof}

The proof of the following theorem relies on the work of Vistoli on the Chow ring of the classifying space of $\PGLp$ \cite{Vistoli}. By a result of Morel-Voevodsky, \cite[Proposition 4.2.6]{MV}, we have
\[
	A^*_G \iso H^{2*} (\Bet G, \bbZ(*)) \; ,
\]
where $G$ is a linear algebraic group, and $A^*_G$ is the Chow ring of the classifying space of $G$, in the sense of Totaro \cite{Totaro}.

\begin{theorem} \label{chow-of-PGLp}
Let $p$ be an odd prime. Over any field of characteristic zero containing a primitive $p^{th}$ root of unity, the canonical homomorphism in motivic cohomology
\[
	H^{2*} (\Bet \PGLp, \bbZ(*)) \to H^{2*} (B\PGLp, \bbZ(*))
\]
is injective.
\end{theorem}

\begin{proof}
In \cite{Vistoli}, Vistoli defines a subgroup $C_P \times \mu_p \subset \PGLp$, as follows.

Let $\omega$ be a primitive $p^{th}$ root of unity in the base field $k$, and let $\tau$ be the diagonal matrix $\diag(\omega, \omega^2, \dots, \omega^{p-1}, 1)$. Then $\tau$ generates a subgroup of $\PGLp$ isomorphic to $\mu_p$.

Let $\sigma$ be the permutation matrix corresponding to the cycle 
$(1 \, 2 \, \dots p) \in S_p$. Then $\sigma$ generates a subgroup of $\PGLp$ isomorphic to $C_p$, the cyclic group of order $p$, viewed as a group scheme over $k$ in the usual way. Of course, 
$C_p \iso \mu_p$, but the notation is meant to be suggestive.

In $\textbf{Gl}_p$, we have $\tau \sigma = \omega \sigma \tau$, so $\sigma$ and $\tau$ commute in $\PGLp$, and they generate a subgroup of $\PGLp$ isomorphic to $C_p \times \mu_p$. 

Just for this proof, write
\[
	CH^*X = H^{2*} (X, \bbZ(*)) \; .
\]

Let $T_{PGl_p}$ be the standard maximal torus in $\PGLp$, consisting of classes of diagonal matrices. By work of Totaro and Vistoli, \cite[Proposition 9.3]{Vistoli} and \cite[Proposition 9.4]{Vistoli}, the inclusions $T_{PGl_p} \subset \PGLp$ and $C_p \times \mu_p \subset \PGLp$ induce an injective homomorphism
\[
	CH^* \Bet \PGLp \to CH^* \Bet T_{PGL_p} 
	\times CH^* \Bet (C_p \times \mu_p) \; .
\]
We have (\cite[p194]{Vistoli}):
\[
	CH^* \Bet (C_p \times \mu_p) \iso
	\bbZ[\xi, \eta] / (p\xi, p\eta) \; ,
\]
where $\xi$ is the first chern class of the character with $\sigma \mapsto \omega$ and $\tau \mapsto 1$, and $\eta$ is the first chern class of the character with $\sigma \mapsto 1$ and $\tau \mapsto \omega$. Using this and an argument analogous to the proof of \ref{chow-of-mu}, one shows that the natural map
\[
	CH^* \Bet (C_p \times \mu_p) \to
	CH^* B (C_p \times \mu_p)
\]
is injective. As group-schemes, we have
\[
	T_{PGl_p} \iso T_{Sl_p} \iso \Gm^{\times p-1} \; ,
\]
so that $\Bet T_{PGL_p} \equiv B T_{PGl_p}$.

Consider the commutative diagram
\begin{footnotesize}
\[
	\xymatrix{
	&CH^* \Bet \PGLp \ar[r] \ar[d]
	&CH^* B\PGLp \ar[d] \\
	&CH^* \Bet T_{PGL_p} \times 
	CH^* \Bet (C_p \times \mu_p) \ar[r]
	&CH^* B T_{PGL_p} \times 
	CH^* B (C_p \times \mu_p)\\
	}
\]
\end{footnotesize}

The bottom route around the square is injective, and it follows that
\[
	CH^* \Bet \PGLp \to CH^* B\PGLp
\]
is injective.
\end{proof}

Finally, let's summarize what \ref{chow-of-PGLp} tells us about $H^{2*} (B\PGLp, \bbZ(*))$.

Let $R = H^* (k, \bbZ(*))$. We have
\[
	H^* (B\Gm^{\times n}, \bbZ(*)) \iso R[x_1, \dots, x_n] \; ,
\]
where $x_i \in H^2 (B\Gm^{\times n}, \bbZ(1))$,
\[
	H^* (B\GLn, \bbZ(*)) \iso R[\sigma_1, \dots, \sigma_n] \; ,
\]
where $\sigma_i$ is the $i^{th}$ elementary symmetric polynomial in the variables $x_i$, and
\[
	H^* (B T_{PGl_n}, \bbZ(*)) \iso 
	R[x_1 - x_2, \dots, x_{n-1} - x_n] \; .
\]
Furthermore, there is a commutative diagram:
\[
	\xymatrix{
	&H^* (B\PGLn, \bbZ(*)) \ar[r] \ar[d] 
	&R[x_1 - x_2, \dots, x_{n-1} - x_n] \ar[d]\\
	&R[\sigma_1, \dots, \sigma_n] \ar[r]
	&R[x_1, \dots, x_n]\\
	}
\]
This gives a homomorphism to the symmetric part of the cohomology of the maximal torus:
\[
	H^* (B\PGLn, \bbZ(*)) \to 
	R[x_1 - x_2, \dots, x_{n-1} - x_n]^{S_n} \; .
\]
Restricting attention to the Chow groups, and letting $n = p$, we have a homomorphism
\[
	H^{2*} (B\PGLp, \bbZ(*)) \to 
	\bbZ[x_1 - x_2, \dots, x_{p-1} - x_p]^{S_p} \; ,
\]
and \cite[Theorem 3.2]{Vistoli}, and \ref{chow-of-PGLp} together imply that this homomorphism has a section. Furthermore, there is an element 
$\rho \in H^{2p+2} (B\PGLp, \bbZ(p+1))$ not in the image of this section, satisfying some relations given in \cite[Theorem 3.3]{Vistoli}.

\bibliographystyle{amsplain}
\bibliography{ref}

\end{document}